\documentclass[11pt]{amsart}
\usepackage{geometry}            
\usepackage{graphicx}
\usepackage{amsmath, mathtools, amsfonts, amssymb}
\usepackage{epstopdf}
\usepackage{epsfig,epsf,xcolor,verbatim}

\usepackage[alphabetic]{amsrefs}

\def\R{\Bbb R}

\def\N{\Bbb N}

\def\H{\Bbb H}
\def\Z{\Bbb Z}

\numberwithin{equation}{section}
\numberwithin{figure}{section}

\renewcommand{\leq}{\leqslant}
\renewcommand{\le}{\leqslant}
\renewcommand{\geq}{\geqslant}
\renewcommand{\ge}{\geqslant}

\newtheorem{theorem}{Theorem}[section]
\newtheorem{lemma}[theorem]{Lemma}
\newtheorem{corollary}[theorem]{Corollary}
\newtheorem{proposition}[theorem]{Proposition}
\newtheorem*{theorem*}{Theorem}
\newtheorem*{corollary*}{Corollary}

\newcommand{\arccosh}{\mathrm{arccosh} \,} 

\title{Napoleonic Constructions in the Hyperbolic Plane}

\author{Serena Dipierro}
\author{Lyle Noakes}
\author{Enrico Valdinoci}

\address{Department of Mathematics and Statistics,
University of Western Australia, 35 Stirling Highway,
WA6009 Crawley, Australia}
\email{serena.dipierro@uwa.edu.au, lyle.noakes@uwa.edu.au, enrico.valdinoci@uwa.edu.au}

\begin{document}
\maketitle

\begin{abstract}
In the Euclidean setting, Napoleon's Theorem states that if one constructs an equilateral triangle on either the outside or the inside of each side of a given triangle and then connects the barycenters of those three new triangles, the resulting triangle
happens to be equilateral.

The case of spherical triangles has been recently shown to be different: on the sphere,
besides equilateral triangles, a necessary and sufficient condition for a given triangle
to enjoy the above Napoleonic property is that its congruence class should lie on a suitable surface
(namely, an ellipsoid in suitable coordinates).

In this article we show that the hyperbolic case is significantly different from both the Euclidean and the spherical setting. Specifically, we establish here that the hyperbolic plane does not admit any Napoleonic triangle, except the equilateral ones. Furthermore, we prove that iterated Napoleonization of any triangle causes it to 
become smaller and smaller, more and more
equilateral and converge to a single point in the limit.
\end{abstract}

\section{Introduction}

The {\em Napoleonic construction} deals with triangles on a surface and
proceeds according to the following steps:
\begin{itemize} \item A triangle~$P_0P_1P_2$ in a two-dimensional surface is given.
\item Three equilateral triangles are constructed on the sides of~$P_0P_1P_2$: namely,
one takes points~$Q_0$, $Q_1$ and~$Q_2$ on the surface such that
the triangles~$P_0P_1Q_2$, $P_1P_2Q_0$ and~$P_0P_2Q_1$ are equilateral
(two different constructions arise, according to the direction chosen for the points~$Q_0$, $Q_1$ and~$Q_2$).
\item The centroids (i.e. barycenters) $R_0$, $R_1$ and~$R_2$ of the equilateral triangles~$P_1P_2Q_0$, $P_0P_2Q_1$ and~$P_0P_1Q_2$ are considered and
the triangle~$R_0R_1R_2$ is called the {\em{Napoleonization}} of~$P_0P_1P_2$.
\item If the Napoleonization~$R_0R_1R_2$ is an equilateral triangle, then the initial triangle~$P_0P_1P_2$ is called {\em Napoleonic}.
\end{itemize}

The classical case of this construction occurs when the ambient surface is the {\em Euclidean plane}. In this situation, {\em all triangles are Napoleonic}: this is a famous result going under the name of Napoleon's Theorem: see e.g.~\cite{MR2928662} and the references therein for the fascinating history of this result (see also Figures~\ref{FIG1} and~\ref{FIG2} for a sketch of the Napoleonic constructions in the Euclidean plane).

\begin{figure}[t]
\centering
\includegraphics[width=0.3\textwidth]{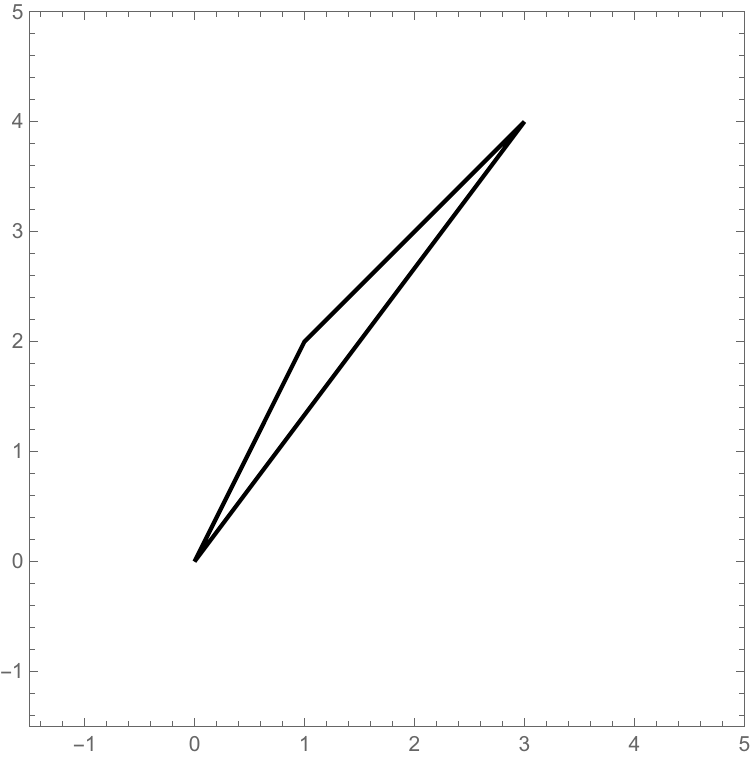}$\;$
\includegraphics[width=0.3\textwidth]{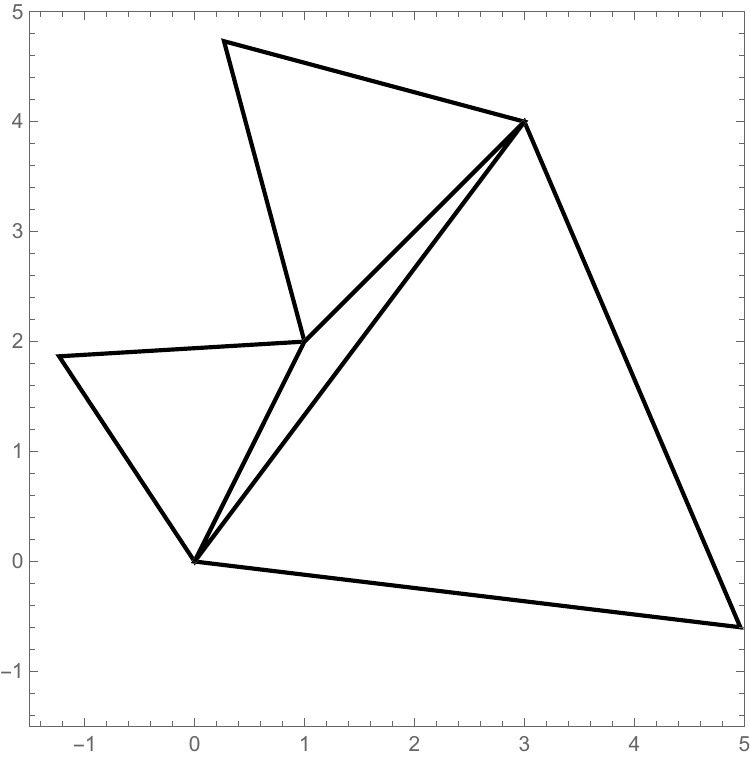}$\;$
\includegraphics[width=0.3\textwidth]{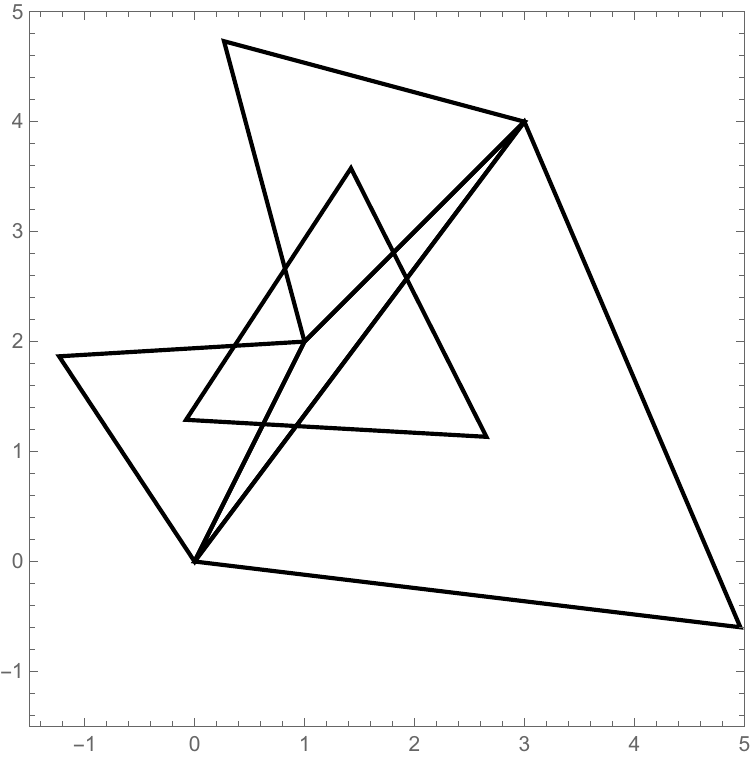}
\caption{External Napoleonization of the Euclidean triangle with vertices in~$(0,0)$, $(1,2)$ and~$(3,4)$.}\label{FIG1}
\end{figure}

In the Euclidean case, Napoleon's Theorem attracted the attention of several first-rate mathematicians, including Fields Medallist Jesse Douglas; in fact, the question of
extending Napoleon's Theorem from planar triangles to polygons is known as the
Petr-Douglas-Neumann problem, see~\cite{zbMATH02641104, MR2178, MR6839}.
Napoleon's Theorem also finds practical applications in some optimization questions,
such as the Fermat-Steiner-Torricelli problem, see~\cite{MR4298718}.

\begin{figure}[b]
\centering
\includegraphics[width=0.3\textwidth]{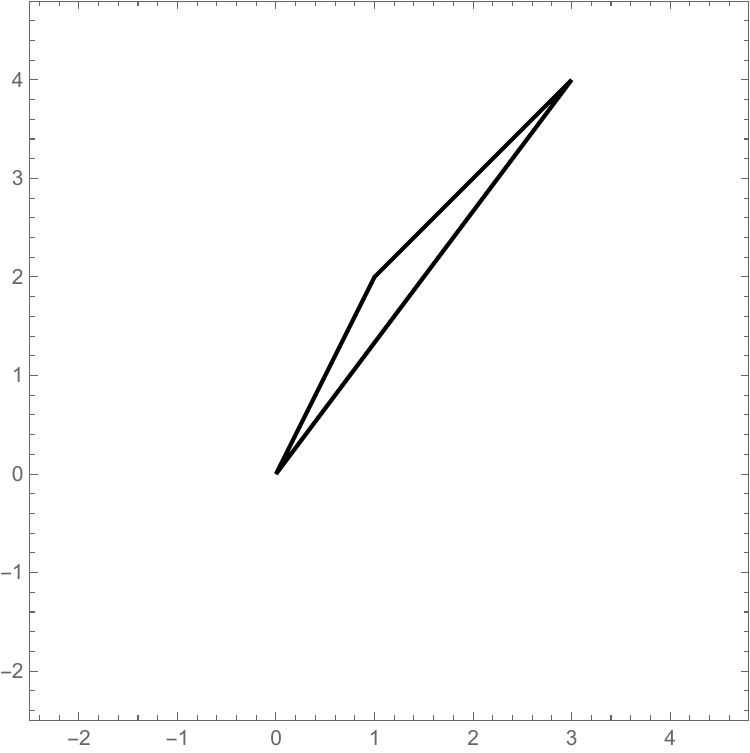}$\;$
\includegraphics[width=0.3\textwidth]{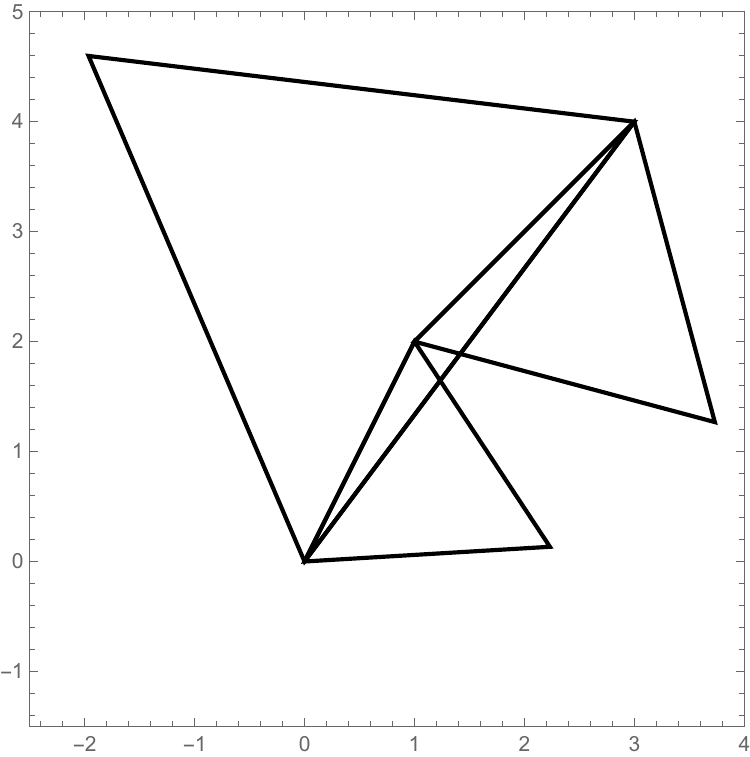}$\;$
\includegraphics[width=0.3\textwidth]{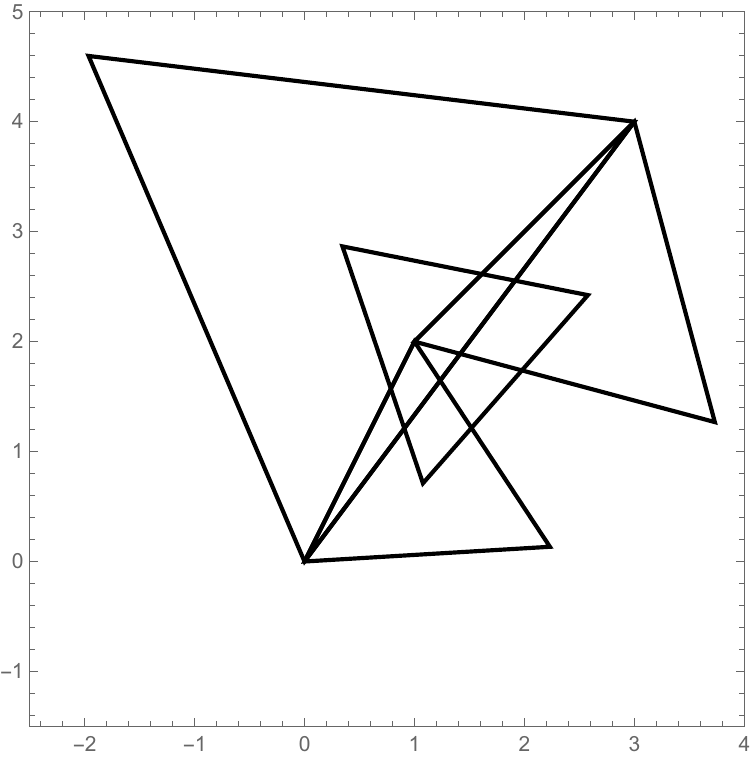}
\caption{Internal Napoleonization of the Euclidean triangle with vertices in~$(0,0)$, $(1,2)$ and~$(3,4)$.}\label{FIG2}
\end{figure}

Among the many modern extensions of Napoleon's Theorem, a natural field of investigation is the discovery and classification of Napoleonic triangles in ambient surfaces different from the Euclidean plane. This happens to be a rather difficult problem about which little is known.

So far, the only ambient manifold for which all Napoleonic triangles have been classified is the {\em round sphere}. Specifically, in~\cite{MR4727700} it is established that if a spherical triangle is Napoleonic then either it is equilateral or its 
congruency class lies in an explicit surfaces, which, in an appropriate coordinate system, can be written as a {\em two-dimensional rotational ellipsoid} (and, conversely, all 
congruency classes in this ellipsoid 
correspond to Napoleonic triangles). What is more, {\em all non-equilateral Napoleonic triangles on the sphere produce congruent Napoleonizations}. 

The case of Napoleonic constructions in the {\em hyperbolic plane} appears then as a natural question. So far, to the best of our knowledge, the only available results in this setting go back to~\cite{MR1847491}
and deal with an infinite sequence of recursively
defined hyperbolic triangles (as explicitly mentioned in~\cite{MRev}
this construction is inspired by Napoleon's Theorem, but structurally quite different from it). In particular, no specific investigation of Napoleonic constructions in the hyperbolic plane has been carried out till now.

The goal of this paper is to fill this gap.
For concreteness, as a model for the hyperbolic plane
we consider here the upper sheet of the unit hyperboloid
in the Minkowski space (see Section~\ref{MINK} for details).
Our first result states that the only Napoleonic triangles are the trivial ones
(i.e., the ones for which the initial triangle was equilateral):

\begin{theorem}\label{non-exis}
If the Napoleonization of a hyperbolic triangle is equilateral,
then the initial triangle is equilateral too.
\end{theorem}

We stress that the hyperbolic case dealt with in Theorem~\ref{non-exis} is surprisingly different both from the Euclidean case (in which all Napoleonizations are equilateral)
and the spherical case (in which an ellipsoid of parameters produces non-trivial cases
of equilateral Napoleonizations). We believe that the structural differences of
Napoleon-like results depending on the geometry of the ambient surface is indeed a noteworthy phenomenon and a brand-new line of investigation which deserves
a deeper understanding.\medskip

Another question of interest in this setting is what happens after repeated Napoleonizations, i.e. by taking a hyperbolic triangle to start with, and applying
the Napoleonic construction over and over. 
For that target, given a hyperbolic triangle~$P_0P_1P_2$, we denote
its Napoleonization by~$P_0^{(1)}P_1^{(1)}P_2^{(1)}$, and recursively we set~$P_0^{(k)}P_1^{(k)}P_2^{(k)}$ to be the Napoleonization of~$P_0^{(k-1)}P_1^{(k-1)}P_2^{(k-1)}$.
In this framework, our result
goes as follows:

\begin{theorem}\label{REPENE}
As~$k$ increases, the triangles~$P_0^{(k)}P_1^{(k)}P_2^{(k)}$ become smaller, more
nearly equilateral and, as~$k\rightarrow +\infty$, more nearly a single point.  
\end{theorem}

We point out that Theorem~\ref{REPENE} is in sharp contrast with the Euclidean case
(flat triangles remain equilateral and do not contract under repeated Napoleonizations, actually they just rotate by~$60^\circ$, up to relabeling vertices).
The comparison with the Euclidean case also highlights an unavoidable difficulty
intrinsically linked to the proof of Theorem~\ref{REPENE}: indeed, if
repeated Napoleonizations tend to approach a point, the setting becomes ``more and more Euclidean'' during the iteration, thus making the convergence to a point problematic precisely when we approach the limit.

\medskip

The rest of the paper is organized as follows. 
In Section~\ref{MINK} we gather some preliminary observations
on the hyperboloid model and the hyperbolic triangles.
In Section~\ref{u4i3tifgesec:bespoke}
we introduce a bespoke set of hyperbolic coordinates, which are different
from the standard hyperbolic distance~$\arccosh(-\langle\cdot,\cdot\rangle)$
and come in handy to simplify several otherwise cumbersome calculations.
Sections~\ref{sec:profthem1} and~\ref{sec:proofthm2}
contain the proofs of Theorems~\ref{non-exis}
and~\ref{REPENE}, respectively.

\section{Preliminaries on the Hyperboloid Model and the Hyperbolic Triangles}\label{MINK}

We recall that the Minkowski inner product~$\langle ~,~\rangle$ on~$\R ^3$ is given by 
$$\langle (v_1,v_2, v_3)^{\bf T},(w_1,w_2,w_3)^{\bf T}\rangle :=-v_1w_1+v_2w_2+v_3w_3.$$
Equivalently \begin{equation*}
\langle v,w\rangle = (Jv)\cdot w=v\cdot (Jw),\end{equation*} where~$\cdot$ denotes the Euclidean inner product and~$J$ is the~$3\times 3$ 
diagonal matrix
$$\left[ \begin{array}{ccc}
-1&0&0\\
0&1&0\\
0&0&1
\end{array}\right] .$$
The (upper unit) hyperboloid  is  
$$\H := \big\{ P\in \R ^3\;{\mbox{ s.t. }}\;\langle P,P\rangle =-1\;{\mbox{ and }}\; \langle P,E_1\rangle \geq 1\big\} $$
where~$E_1$ is the timelike vector~$(-1,0,0)^{\bf T}$. Moreover, we denote by~$E_2:=(0,1,0)^{\bf T}$ and~$ E_3:=(0,0,1)^{\bf T}$. 

As customary\footnote{For an elementary proof of~\eqref{csinq}, one can write~$P=(p,\tilde P)$ and~$Q=(q,\tilde Q)$,
with~$p$, $q\in[1,+\infty)$ and~$\tilde P$, $\tilde Q\in\R^2$, and use
the standard Cauchy-Schwarz inequality.
Indeed, we note that, for all~$a$, $b\in\R$,
$$2pqab\le p^2a^2+q^2b^2.$$
Also, choosing~$a:=\sqrt{\tilde Q\cdot \tilde Q}$ and~$b:=\sqrt{\tilde P\cdot \tilde P}$,
we see that~$1=-\langle\tilde P, \tilde P\rangle=p^2-b^2$, and similarly~$1=q^2-a^2$, yielding that~$p\ge b$ and~$q\ge a$ and that
\begin{eqnarray*}
&&1=(p^2-b^2)(q^2-a^2)=p^2q^2+a^2b^2-p^2a^2-q^2b^2\le
p^2q^2+a^2b^2-2pqab=(pq-ab)^2.
\end{eqnarray*}
Consequently, $pq\ge ab$ and~$pq-ab\ge1$.
Thus, since
\begin{eqnarray*}&&
-\langle P,Q\rangle=
pq-\tilde P\cdot \tilde Q
\ge pq-\sqrt{(\tilde P\cdot \tilde P)(\tilde Q\cdot \tilde Q)}=
pq-ab,
\end{eqnarray*}
the classical inequality in~\eqref{csinq} plainly follows
(and actually, tracing the equality cases in the above inequality, one also gets
that equality in~\eqref{csinq} holds if and only if~$P=Q$).} we observe that if~$P$, $Q\in\H$, then,
\begin{equation}\label{csinq}
\langle P,Q\rangle\le-1. \end{equation}

The hyperbolic cross-product~$\tilde \times$ of vectors in Minkowski space~$\R ^3$ is given by 
$$v\tilde \times w :=J(v\times w)$$
where~$\times $ is the Euclidean cross-product. 

Since~$J$ is orthogonal with determinant~$-1$, we deduce from the formula
$$ (Jv)\times (Jw)= {\rm{det}}(J)  (J^{-1})^T (v\times w) =-J(v\times w)$$ that, for any~$v$, $w$, $v'$, $w'\in \R ^3$,
$$\langle v\tilde \times w,v'\tilde \times w'\rangle = 
J(v\times w) \cdot (v'\tilde \times w')=-((Jv)\times (Jw))\cdot(v'\times w').$$
Therefore, using the Binet-Cauchy Identity,
\begin{equation}\label{eq0}\begin{split}
\langle v\tilde \times w,v'\tilde \times w'\rangle 
&=- ((Jv)\cdot v') ((Jw)\cdot w') + ((Jv)\cdot w') ((Jw)\cdot v')
\\&=-\langle v,v'\rangle \langle w,w'\rangle +\langle v,w'\rangle \langle w,v'\rangle .
\end{split}\end{equation}
The hyperbolic scalar triple product~$\langle u, v\tilde \times w\rangle$ of~$u$, $v$, $w\in \R ^3$ is also the Euclidean scalar triple product~$u\cdot(v\times w)$, so it has the same symmetries.

Given two points on the hyperboloid, a third point can be written as a combination
of these two points and their cross-product. More explicitly, we have that:

\begin{lemma}\label{GT}
Let~$P_0\not= P_1\in \H$. Then, any~$Q\in \H$ can be written in the form 
\begin{equation}\label{QDE}
Q=a_0P_0+a_1P_1+bP_0\tilde \times P_1,\end{equation}
with
\begin{equation}\label{eq20}-a_0=\frac{\langle Q,P_0\rangle +\langle P_0,P_1\rangle \langle Q,P_1\rangle }{1-\langle P_0,P_1\rangle ^2}
\quad{\mbox{and}}\quad
-a_1=\frac{\langle Q,P_1\rangle+\langle P_0,P_1\rangle \langle Q,P_0\rangle}{1-\langle P_0,P_1\rangle ^2}.\end{equation} 
Also,
\begin{equation}\label{eq1}-1=\langle Q,Q\rangle = -a_0^2-a_1^2+2a_0a_1\langle P_0,P_1\rangle -b^2(1-\langle P_0,P_1\rangle ^2).\end{equation}
\end{lemma} 

\begin{proof} We point out that, in light of~\eqref{csinq} and the fact that~$P_0\not= P_1$,
$$ 1-\langle P_0,P_1\rangle ^2\neq 0,$$
and therefore the coefficients in~\eqref{eq20} are well defined.

Now, using the facts that~$\langle P_0,P_0\rangle=-1=\langle P_1,P_1\rangle$ 
and that~$\langle P_0,P_0\tilde \times P_1\rangle= 0=\langle P_1,P_0\tilde \times P_1\rangle$
and~\eqref{eq0}, we see that
\begin{eqnarray*}
\langle Q,Q\rangle &=&\langle a_0P_0+a_1P_1+bP_0\tilde \times P_1, a_0P_0+a_1P_1+bP_0\tilde \times P_1\rangle \\&=&
-a_0^2-a_1^2+2a_0a_1 \langle P_0,P_1\rangle +b^2\langle P_0\tilde \times P_1, P_0\tilde \times P_1\rangle\\&=&
-a_0^2-a_1^2+2a_0a_1 \langle P_0,P_1\rangle +b^2
(\langle P_0, P_1\rangle^2-1),
\end{eqnarray*}
which gives the claim in~\eqref{eq1}.

Thus, noticing that
\begin{eqnarray*}&&
-a_0+a_1\langle P_0,P_1\rangle =\langle Q,P_0\rangle \\ {\mbox{and }}&&
a_0\langle P_0,P_1\rangle -a_1=\langle Q,P_1\rangle ,
\end{eqnarray*}
we obtain~\eqref{eq20}, as desired.\end{proof}

The case of isosceles and equilateral triangles on the hyperboloid are particular cases of Lemma~\ref{GT} and go as follows:

\begin{corollary}
Let~$P_0P_1Q$ be isosceles with~$P_0\ne P_1$ and~$\langle P_0,Q\rangle =\langle P_1,Q\rangle$. 

Then, $Q$ can be written as in~\eqref{QDE}, with
\begin{equation}\label{eq2}
-a_0=-a_1=-a:=\frac{\langle P_0,Q\rangle}{1-\langle P_0,P_1\rangle}.
\end{equation}
\end{corollary}

\begin{proof} When~$\langle P_0,Q\rangle =\langle P_1,Q\rangle$,
we have that~\eqref{eq20} reduces to~\eqref{eq2}.
\end{proof}

\begin{corollary}\label{CORO2}
Let~$P_0P_1Q_2$ be equilateral with~$P_0\ne P_1\ne Q_2$. 

Then,
\begin{equation*}
Q_2=\frac{-\langle P_0,P_1\rangle(P_0+P_1)+\epsilon _2\sqrt{1-2\langle P_0,P_1\rangle }P_0\tilde \times P_1}{1-\langle P_0,P_1\rangle}\end{equation*}
with~$\epsilon _2\in\{- 1,1\}$.
\end{corollary}

\begin{proof} We recall that, being~$P_0P_1Q_2$ equilateral, we have that
$$\langle P_0,P_1\rangle =\langle P_0,Q_2\rangle=\langle P_1,Q_2\rangle .$$ 
Therefore, by~\eqref{eq1} and~\eqref{eq2}, used here with~$Q:=Q_2$, we see that
\begin{eqnarray*}&&
b^2(1-\langle P_0,P_1\rangle ^2)
=1-2a^2(1-\langle P_0,P_1\rangle )
\\&&\qquad=1-2\frac{\langle P_0,P_1\rangle ^2}{1-\langle P_0,P_1\rangle}=\frac{(1-2\langle P_0,P_1\rangle )(1+\langle P_0,P_1\rangle )}{1-\langle P_0,P_1\rangle}.
\end{eqnarray*}
As a consequence,
$$
b=\epsilon _2\frac{\sqrt{1-2\langle P_0,P_1\rangle }}{1-\langle P_0,P_1\rangle} $$
and the desired result follows from~\eqref{QDE}.\end{proof}

We now reconsider Corollary~\ref{CORO2} with the aim of identifying the
centroid of an equilateral triangle on the unit hyperboloid
(where, by definition, the centroid of a triangle is the sum of the coordinates of its vertices projected over the hyperboloid):

\begin{lemma}\label{ILLEc}
Let~$P_0P_1Q_2$ be equilateral with~$P_0\ne P_1\ne Q_2$. 
Let~$R_2\in\H$ be its centroid.

Then, 
$$R_2=\frac{\sqrt{1-2\langle P_0,P_1\rangle }(P_0+P_1)+\epsilon _2P_0\tilde \times P_1}{\sqrt{3}(1-\langle P_0,P_1\rangle )}$$
with~$\epsilon _2\in\{- 1,1\}$.
\end{lemma}

\begin{proof}
The centroid~$R_2$ of the equilateral hyperbolic triangle~$P_0P_1Q_2$ is~$\hat R_2/\sqrt {-\langle \hat R_2,\hat R_2\rangle}$, where  
$$\hat R_2:=(1-2\langle P_0,P_1\rangle )(P_0+P_1)+\epsilon _2\sqrt{1-2\langle P_0,P_1\rangle }P_0\tilde \times P_1 $$
and, thanks to the equality in~\eqref{eq0},
\begin{eqnarray*}&&
-\langle \hat R_2,\hat R_2\rangle \\
&=&2(1-2\langle P_0, P_1\rangle)^2 - 2(1-2\langle P_0, P_1\rangle)^2\langle P_0, P_1\rangle+(1-2\langle P_0, P_1\rangle)(1-\langle P_0, P_1\rangle^2)\\
&=&(1-2\langle P_0, P_1\rangle) \big(2(1-2\langle P_0, P_1\rangle)-2 (1-2\langle P_0, P_1\rangle)\langle P_0, P_1\rangle +1-\langle P_0, P_1\rangle^2\big)
\\&=&3(1-2\langle P_0,P_1\rangle )(1-\langle P_0,P_1\rangle )^2.
\end{eqnarray*}
The desired result now plainly follows.
\end{proof}

\section{A Bespoke Set of Hyperbolic Coordinates}\label{u4i3tifgesec:bespoke}

{F}rom now on we consider three distinct points~$P_0$, $P_1$, $P_2\in\H$ and define
\begin{equation}\label{firstalp}\alpha :=-1+ \langle P_0,P_1\rangle + \langle P_1,P_2\rangle +\langle P_2,P_0\rangle \end{equation} and
\begin{equation}\label{firstalpBIS}
\chi  :=\langle P_0\tilde \times P_1, P_2\rangle .\end{equation}
We stress that~$\alpha$ is symmetric with 
respect to permutations of~$P_0$, $P_1$ and~$P_2$, and that~$\chi$  is symmetric with respect to cyclic permutations. 

Since Theorems~\ref{non-exis} and~\ref{REPENE} are invariant under permutations of~$P_0$, $P_1$ and~$P_2$, we can
list the vertices~$P_0$, $P_1$ and~$P_2$ so that
\begin{equation}\label{chimag}
\chi\ge0.
\end{equation}
This reordering of vertices will be implicitly assumed in what follows.

{F}rom now on, we will also consider~$Q_0$, $Q_2\in\H$ such that~$P_0P_1Q_2$ and~$P_1P_2Q_0$
are equilateral. Let also~$R_0$ be the centroid of~$P_1P_2Q_0$
and~$R_2$ be the centroid of~$P_0P_1Q_2$.

With this notation, we can take a step further from Lemma~\ref{ILLEc} and obtain that:

\begin{lemma}
We have that
\begin{equation*}\begin{split}&3(1-\langle P_0,P_1\rangle )(1-\langle P_1,P_2\rangle )\langle R_2,R_0\rangle \\&\qquad=\alpha
\sqrt{1-2\langle P_0,P_1\rangle }\sqrt{1-2\langle P_1,P_2\rangle } +\chi
\Big(\epsilon _0 \sqrt{1-2\langle P_0,P_1\rangle } +
\epsilon _2 \sqrt{1-2\langle P_1,P_2\rangle }\Big)  \\&\qquad\qquad -\epsilon _2\epsilon _0(\langle P_0,P_1\rangle \langle P_1,P_2\rangle +\langle P_0,P_2\rangle),\end{split}\end{equation*}with~$\epsilon_0$, $\epsilon _2\in\{- 1,1\}$.
\end{lemma}

\begin{proof} By swapping indexes in Lemma~\ref{ILLEc}, we see that
the equilateral hyperbolic triangle~$P_1P_2Q_0$ has centroid 
$$R_0=\frac{\sqrt{1-2\langle P_1,P_2\rangle }(P_1+P_2)+\epsilon _0P_1\tilde \times P_2}{\sqrt{3}(1-\langle P_1,P_2\rangle )}.$$
Therefore, exploiting~\eqref{eq0}
and the cyclic symmetry of the hyperbolic scalar triple product,
\begin{eqnarray*}&&3(1-\langle P_0,P_1\rangle )(1-\langle P_1,P_2\rangle )\langle R_2,R_0\rangle \\&&\qquad=
\langle \sqrt{1-2\langle P_0,P_1\rangle }(P_0+P_1)+\epsilon _2P_0\tilde \times P_1, \sqrt{1-2\langle P_1,P_2\rangle }(P_1+P_2)+\epsilon _0P_1\tilde \times P_2\rangle \\&&\qquad=
\sqrt{1-2\langle P_0,P_1\rangle }\sqrt{1-2\langle P_1,P_2\rangle }\langle P_0+P_1, P_1+P_2\rangle \\&&\qquad\qquad+
\epsilon _0 \sqrt{1-2\langle P_0,P_1\rangle }\langle P_0+P_1, P_1\tilde \times P_2\rangle +
\epsilon _2 \sqrt{1-2\langle P_1,P_2\rangle }\langle P_0\tilde \times P_1, P_1+P_2\rangle \\&&\qquad\qquad+
\epsilon _2\epsilon _0\langle P_0\tilde \times P_1, P_1\tilde \times P_2\rangle \\&&\qquad=
\sqrt{1-2\langle P_0,P_1\rangle }\sqrt{1-2\langle P_1,P_2\rangle }(-1+\langle P_0,P_1\rangle +\langle P_1,P_2\rangle +\langle P_2,P_0\rangle ) \\
&&\qquad\qquad+
(\epsilon _0 \sqrt{1-2\langle P_0,P_1\rangle } +
\epsilon _2 \sqrt{1-2\langle P_1,P_2\rangle })\langle P_0\tilde \times P_1, P_2\rangle 
\\&&\qquad\qquad
 -\epsilon _2\epsilon _0(\langle P_0,P_1\rangle \langle P_1,P_2\rangle +\langle P_0,P_2\rangle).\end{eqnarray*}
{F}rom this the desired result follows.
\end{proof}

We now introduce a new set of ``hyperbolic units of measurements'', technically and conceptually different from the standard hyperbolic distance~$\arccosh(-\langle\cdot,\cdot\rangle)$,
which come in handy to simplify several otherwise cumbersome calculations.
Namely, we define  
\begin{equation}\label{cooR}\begin{split}&d_0:=\sqrt{1-2\langle P_1,P_2\rangle},
\\&d_1:=\sqrt{1-2\langle P_2,P_0\rangle}\\
{\mbox{and }}\quad& d_2:=\sqrt{1-2\langle P_0,P_1\rangle} .\end{split}\end{equation}
We point out that $(d_0,d_1,d_2)$ defines the congruency class of the hyperbolic triangle, and that these hyperbolic coordinates are well defined, thanks
to~\eqref{csinq}.

The main features of these hyperbolic coordinates are the following:

\begin{lemma}
For all~$i\in \Z/3\Z$, 
we have that 
\begin{equation}\label{triineq.2}d_i\geq \sqrt{3}\end{equation}
and
\begin{equation}\label{triineq}
d_i^2-1\leq (d_{i+1}^2-1)(d_{i+2}^2-1).
\end{equation}
\end{lemma}

\begin{proof} The claim in~\eqref{triineq.2} follows directly from~\eqref{csinq}.

Also, by the triangle inequality in~$\H$ (see e.g. page~70 in~\cite{MR1205776}), for~$i\in \Z/3\Z$, 
$$\arccosh \left(\frac{d_i^2-1}{2}\right)\leq \arccosh \left(\frac{d_{i+1}^2-1}{2}\right)+\arccosh \left(\frac{d_{i+2}^2-1}{2}\right).$$

We recall that
$$ \cosh \big( \arccosh x+\arccosh y\big) =xy+\sqrt{(x^2-1)(y^2-1)}.$$
Therefore, using this formula with~$x:=\frac{d_{i+1}^2-1}{2}$
and~$y:=\frac{d_{i+2}^2-1}{2}$, we find that
\begin{eqnarray*}
2d_i^2-2&\leq& (d_{i+1}^2-1)(d_{i+2}^2-1)+\sqrt{ (d_{i+1}^4-2d_{i+1}^2-3)(d_{i+2}^4-2d_{i+2}^2-3)}\\&\le&2 (d_{i+1}^2-1)(d_{i+2}^2-1),
\end{eqnarray*}
from which one obtains~\eqref{triineq}.
\end{proof}

Now we calculate~$\alpha$ and~$\chi$:

\begin{proposition}
We have that 
\begin{equation}
\label{firstalphaeq} 2\alpha =1-d_0^2-d_1^2-d_2^2\end{equation}
and
\begin{equation}
\label{chieq}2\chi =\sqrt{3(d_0^2+d_1^2+d_2^2)-(d_0^2d_1^2+d_1^2d_2^2+d_2^2d_0^2+d_0^4+d_1^4+d_2^4)+d_0^2d_1^2d_2^2}.
\end{equation}\end{proposition}

\begin{proof} The claim in~\eqref{firstalphaeq} follows from~\eqref{firstalp} and~\eqref{cooR}.


To calculate~$\chi$ note first that, by Lemma~\ref{GT},  
$$P_2= -\frac{\langle P_2,P_0\rangle +\langle P_0,P_1\rangle \langle P_1,P_2\rangle }{1-\langle P_0,P_1\rangle ^2}P_0-\frac{\langle P_1,P_2\rangle+\langle P_0,P_1\rangle \langle P_2,P_0\rangle}{1-\langle P_0,P_1\rangle ^2}P_1+bP_0\tilde \times P_1.$$
Hence, using also~\eqref{eq0},
$$\chi = \langle P_0\tilde \times P_1,P_2\rangle =b\langle P_0\tilde \times P_1,P_0\tilde \times P_1\rangle =-b(1-\langle P_0,P_1\rangle ^2).$$

Moreover, by~\eqref{eq1}, \begin{eqnarray*}
&&b^2(1-\langle P_0,P_1\rangle ^2)^3\\
&=&(1-\langle P_0,P_1\rangle ^2)^2(1-a_0^2-a_1^2+2a_0a_1\langle P_0,P_1\rangle )\\&=& (1-\langle P_0,P_1\rangle ^2)^2-(\langle P_2,P_0\rangle +\langle P_0,P_1\rangle \langle P_1,P_2\rangle )^2-(\langle P_1,P_2\rangle+\langle P_0,P_1\rangle \langle P_2,P_0\rangle)^2\\&& \qquad+
2(\langle P_2,P_0\rangle +\langle P_0,P_1\rangle \langle P_1,P_2\rangle )(\langle P_1,P_2\rangle+\langle P_0,P_1\rangle \langle P_2,P_0\rangle)\langle P_0,P_1\rangle \\
&=&(1-\langle P_0,P_1\rangle ^2)^2
+2\langle P_0,P_1\rangle^3 \langle P_0,P_2\rangle \langle P_1,P_2\rangle   -2\langle P_0,P_1\rangle \langle P_0,P_2\rangle \langle P_1,P_2\rangle\\&&\qquad
+\langle P_0,P_1\rangle^2 \langle P_0,P_2\rangle^2 +\langle P_0,P_1\rangle^2 \langle P_1,P_2\rangle^2
-  \langle P_0,P_2\rangle^2 -\langle P_1,P_2\rangle^2
\\&=&
(1 -\langle P_0,P_1\rangle ^2)(1 - \langle P_0,P_1\rangle ^2 - \langle P_1,P_2\rangle ^2 - \langle P_2,P_0\rangle ^2 - 2\langle P_0,P_1\rangle 
\langle P_1,P_2\rangle \langle P_2,P_0\rangle  ).\end{eqnarray*}

Thus, we find that
\begin{eqnarray*}\chi & =&-b(1-\langle P_0,P_1\rangle ^2)\\&=&\pm \sqrt{1-\langle P_0,P_1\rangle ^2-\langle P_1,P_2\rangle ^2-\langle P_2,P_0\rangle ^2-2 \langle P_0,P_1\rangle \langle P_1,P_2\rangle \langle P_2,P_0\rangle}\end{eqnarray*}
and then~\eqref{chieq} follows from~\eqref{chimag}.\end{proof}

Now we define
\begin{equation}\label{DEGAm}
\gamma := 3(d_0^2+1)(d_1^2+1)(d_2^2+1)\end{equation}
and we have the following two estimates:

\begin{lemma} It holds that
\begin{eqnarray}\label{chiest}&&
(2\chi )^2\leq\frac{1}{3}\sum_{i=0}^2d_i^2(d_{i+1}^2-3)(d_{i+2}^2-3)\\ {\mbox{and }}&&
\label{alphachigam}
-24\alpha \chi \leq \gamma .
\end{eqnarray}\end{lemma}

\begin{proof} By~\eqref{chieq} and the standard Cauchy-Schwarz inequality
applied to the $3$-dimensional vectors~$(d_0^2,d_1^2,d_2^2)$
and~$(d_1^2,d_2^2,d_0^2)$, we see that
\begin{eqnarray*}
(2\chi )^2&=&3(d_0^2+d_1^2+d_2^2)-(d_0^2d_1^2+d_1^2d_2^2+d_2^2d_0^2+d_0^4+d_1^4+d_2^4)+d_0^2d_1^2d_2^2\\& \leq &
3(d_0^2+d_1^2+d_2^2)-2(d_0^2d_1^2+d_1^2d_2^2+d_2^2d_0^2)+d_0^2d_1^2d_2^2\\& =&
\frac{1}{3}\sum_{i=0}^2d_i^2(d_{i+1}^2-3)(d_{i+2}^2-3).\end{eqnarray*}
This proves~\eqref{chiest}.

To prove~\eqref{alphachigam},
we recall~\eqref{firstalphaeq} and~\eqref{chieq}
and
we write \begin{eqnarray*}-24\alpha \chi& =&6(-2\alpha )(2\chi)\\&\leq& 3((2\alpha )^2+(2\chi )^2)\\&=&
3((d_0^2+d_1^2+d_2^2-1)^2+3(d_0^2+d_1^2+d_2^2)\\&&\qquad-(d_0^2d_1^2+d_1^2d_2^2+d_2^2d_0^2+d_0^4+d_1^4+d_2^4)+d_0^2d_1^2d_2^2)\\&=&
3(d_0^2d_1^2d_2^2+d_0^2d_1^2+d_1^2d_2^2+d_2^2d_0^2+d_0^2+d_1^2+d_2^2+1).\end{eqnarray*}
{F}rom this and~\eqref{DEGAm}, we obtain~\eqref{alphachigam}, as desired.\end{proof}

\section{Napoleonic Triangles and Proof of Theorem~\ref{non-exis}}\label{sec:profthem1}

{F}rom now on, we suppose that~$P_0$, $P_1$ and~$P_2$ are not cogeodesic, namely~$\chi  \not= 0$. Hence, by~\eqref{chimag},
\begin{equation*}
\chi>0.
\end{equation*}
Set also\footnote{Because of these choices, if~$\langle P_2,P_0\rangle =\langle P_0,P_1\rangle$ then~$\langle R_2,R_0\rangle =\langle R_0,R_1\rangle$. Similarly, if~$\langle P_0,P_1\rangle =\langle P_1,P_2\rangle$ then~$\langle R_0,R_1\rangle =\langle R_1,R_2\rangle$, and if~$\langle P_1,P_2\rangle =\langle P_2,P_0\rangle$ then~$\langle R_1,R_2\rangle =\langle R_2,R_0\rangle$.

Our attention is not limited to the case where~$P_0P_1P_2$ is isosceles.} $\epsilon _0=\epsilon _1=\epsilon _2=\epsilon :=\pm1$. 

\begin{lemma}\label{MNEQ}
We have that, for~$i\in \Z/3\Z$, 
\begin{equation}
\label{r1r2}\begin{split}&\gamma \langle R_{i+1},R_{i+2}\rangle\\
&\quad=
(d_i^2+1)\Big( 4\big(\alpha d_{i+1}d_{i+2} +\epsilon \chi  (d_{i+1}+d_{i+2})\big)
- (d_{i+1}^2-1)(d_{i+2}^2-1)+2(d_i^2-1)\Big) \end{split}
\end{equation}
and
\begin{equation}\label{preq}\begin{split}&
\gamma \big(\langle R_{i+2},R_i\rangle - \langle R_i,R_{i+1}\rangle \big)\\&\quad=4(d_{i+2}-d_{i+1})\Big( \alpha (d_0+d_1+d_2-d_0d_1d_2)
+\epsilon \chi (1-d_0d_1-d_1d_2-d_2d_0)\Big) .\end{split}
\end{equation}
\end{lemma}

\begin{proof} 
Using Lemma~\ref{ILLEc} we write that
\begin{eqnarray*}&&
3(1-\langle P_{i+2},P_i\rangle )(1-\langle P_i,P_{i+1}\rangle )
\langle R_{i+1},R_{i+2}\rangle\\&=&
\left\langle
\sqrt{1-2\langle P_{i+2},P_i\rangle }(P_{i+2}+P_i)+\epsilon P_{i+2}\tilde \times P_i,
\sqrt{1-2\langle P_i,P_{i+1}\rangle }(P_i+P_{i+1})
+\epsilon P_i\tilde \times P_{i+1}
\right\rangle .
\end{eqnarray*}
Therefore, recalling also the definitions in~\eqref{firstalp}, \eqref{firstalpBIS} and~\eqref{cooR} and exploiting~\eqref{eq0},
\begin{equation}\label{djiewoygvbdjske7865943}\begin{split}&
3(1-\langle P_{i+2},P_i\rangle )(1-\langle P_i,P_{i+1}\rangle )
\langle R_{i+1},R_{i+2}\rangle\\ =\;&
d_{i+1}d_{i+2}\big(\langle P_0,P_1\rangle
+\langle P_1,P_2\rangle+\langle P_0,P_2\rangle-1\big)
+\epsilon\chi d_{i+1}+\epsilon\chi d_{i+2}\\&\qquad
-\langle P_{i},P_{i+2}\rangle\langle P_i,P_{i+1}\rangle-\langle P_{i+1},P_{i+2}\rangle
\\ =\;&\alpha d_{i+1}d_{i+2}+\epsilon\chi\big(d_{i+1}+d_{i+2}\big)
-\frac{(d_{i+1}^2-1)(d_{i+2}^2-1)}{4}+\frac{d_{i}^2-1}{2}
.\end{split}\end{equation}

Also,
\begin{eqnarray*}
(1-\langle P_{i+2},P_i\rangle )(1-\langle P_i,P_{i+1}\rangle )&=&
\left(1-\frac{1-d_{i+1}^2}{2}\right)
\left(1-\frac{1-d_{i+2}^2}{2}\right)\\
&=&\frac14 (1+d_{i+1}^2) (1+d_{i+2}^2).
\end{eqnarray*}
Plugging this information into~\eqref{djiewoygvbdjske7865943},
we conclude that
\begin{eqnarray*}&&\langle R_{i+1},R_{i+2}\rangle
=\frac1{3(1+d_{i+1}^2) (1+d_{i+2}^2)}\\
&&\qquad\times
\Big(4\big(\alpha d_{i+1}d_{i+2}+\epsilon\chi\big(d_{i+1}+d_{i+2}\big)\big)
-(d_{i+1}^2-1)(d_{i+2}^2-1)+2(d_{i}^2-1)\Big).
\end{eqnarray*}
Thus, the claim in~\eqref{r1r2} follows from
this and the definition of~$\gamma$ in~\eqref{DEGAm}.

Now, to prove~\eqref{preq}, we set 
$$T_i:=(d_{i+1}^2+1)\big(\alpha d_{i+2}d_i +\epsilon \chi  (d_{i+2}+d_i)\big)-(d_{i+2}^2+1)\big(\alpha d_id_{i+1} +\epsilon \chi  (d_i+d_{i+1})\big)$$
and we see that~$T_i$ expands as
\begin{eqnarray*}&&\alpha d_i\big(
d_{i+2}(d_{i+1}^2+1)-d_{i+1}(d_{i+2}^2+1)\big)\\
&&\qquad+\epsilon \chi \big((d_{i+1}^2+1)(d_{i+2}+d_i)-(d_{i+2}^2+1)(d_i+d_{i+1})\big)\\&=&
(d_{i+2}-d_{i+1})\big( \alpha d_i(1-d_{i+1}d_{i+2}) +\epsilon \chi (1-d_0d_1-d_1d_2-d_2d_0)\big) .\end{eqnarray*}
Using this and~\eqref{r1r2}, we find that
\begin{eqnarray*}&&\gamma \big(\langle R_{i+2},R_i\rangle - \langle R_i,R_{i+1}\rangle \big)-4T_i\\&&\qquad=
(d_{i+1}^2+1)\Big( 4\big(\alpha d_{i}d_{i+2} +\epsilon \chi  (d_{i}+d_{i+2})\big)
- (d_{i}^2-1)(d_{i+2}^2-1)+2(d_{i+1}^2-1)\Big)\\&&\qquad\quad
-(d_{i+2}^2+1)\Big( 4\big(\alpha d_{i+1}d_{i} +\epsilon \chi  (d_{i+1}+d_{i})\big)
- (d_{i+1}^2-1)(d_{i}^2-1)+2(d_{i+2}^2-1)\Big)\\&&\qquad\quad
- 4(d_{i+1}^2+1)\big(\alpha d_{i+2}d_i +\epsilon \chi  (d_{i+2}+d_i)\big)+4(d_{i+2}^2+1)\big(\alpha d_id_{i+1} +\epsilon \chi  (d_i+d_{i+1})\big)
\\&&\qquad=
-\left((d_{i+1}^2+1)(d_{i+2}^2-1) -(d_{i+2}^2+1)(d_{i+1}^2-1)\right) (d_i^2-1)+2(d_{i+1}^4-d_{i+2}^4) \\&&\qquad=
-2(d_{i+2}^2  -d_{i+1}^2) (d_i^2-1)+2(d_{i+1}^4-d_{i+2}^4)
\\&&\qquad=
-2(d_{i+2}^2-d_{i+1}^2)(d_0^2+d_1^2+d_2^2-1)\\&&\qquad
=4\alpha (d_{i+2}^2-d_{i+1}^2).\end{eqnarray*}
{F}rom these considerations, we obtain~\eqref{preq}.
\end{proof}

\begin{corollary} Suppose that~$P_0P_1P_2$ is not equilateral.

Then, $R_0R_1R_2$ is equilateral if and only if
\begin{equation}\label{r0r1eqg}  \alpha (d_0+d_1+d_2-d_0d_1d_2) +\epsilon \chi   (1-d_0d_1-d_1d_2-d_2d_0)
 = 0. \end{equation}\end{corollary}
 
\begin{proof} By Lemma~\ref{MNEQ}, 
we have that if~\eqref{r0r1eqg} is satisfied then~$R_0R_1R_2$ is equilateral.

Conversely, if~$R_0R_1R_2$ is equilateral, then~$\langle R_{i+2},R_i\rangle =\langle R_i,R_{i+1}\rangle $ for~$i\in \Z/3\Z$. In light of Lemma~\ref{MNEQ},
this is true if and only if
either~$d_{i+1}=d_{i+2}$ or~\eqref{r0r1eqg} holds true.
Notice that it cannot be that~$d_{i+1}=d_{i+2}$ for all~$i\in \Z/3\Z$,
since~$P_0P_1P_2$ is not equilateral. Therefore, in this case, we see that~\eqref{r0r1eqg} must be satisfied.
\end{proof}

With this preliminary work, we can now complete the proof of the non-existence of non-trivial Napoleonic triangles in the hyperbolic plane.

\begin{proof}[Proof of Theorem~\ref{non-exis}]
Suppose that~$P_0P_1P_2$ is Napoleonic, i.e.~$R_0R_1R_2$ is equilateral.
If~$P_0P_1P_2$ is not equilateral then~\eqref{r0r1eqg} implies that
\begin{equation}\label{algeq}
\alpha ^2(d_0+d_1+d_2-d_0d_1d_2)^2-\chi  ^2 (1-d_0d_1-d_1d_2-d_2d_0)^2=0
,\end{equation} 
with~$\alpha$ and~$\chi ^2$ given by~\eqref{firstalphaeq} and~\eqref{chieq} as symmetric polynomials in~$d_0$, $d_1$ and~$d_2$. After these substitutions, 
one sees that the left-hand side of~\eqref{algeq} is{\footnotesize
\begin{eqnarray*}&&\frac14 \Big( (1-d_0^2-d_1^2-d_2^2)^2(d_0+d_1+d_2-d_0d_1d_2)^2\\&&\quad
- \Big(3(d_0^2+d_1^2+d_2^2)-(d_0^2d_1^2+d_1^2d_2^2+d_2^2d_0^2+d_0^4+d_1^4+d_2^4)+d_0^2d_1^2d_2^2\Big)
(1-d_0d_1-d_1d_2-d_2d_0)^2\Big),
\end{eqnarray*}}
that in turn equals to
$$
\frac{\gamma}{24}\Big((d_0-d_1)^2+(d_1-d_2)^2+(d_2-d_0)^2\Big)\Big(d_0^2+d_1^2+d_2^2+d_0d_1+d_1d_2+d_2d_0-2\Big),
$$ which is strictly positive.

Accordingly, we have that~$P_0P_1P_2$ must be equilateral.
In conclusion, all Napoleonic triangles in~$\H$ are equilateral, as 
claimed. \end{proof}

\section{Napoleonic Progressions and proof of Theorem~\ref{REPENE}}
\label{sec:proofthm2}

{F}rom now on, we deal with repeated Napoleonization.
For this purpose, we use the notation introduced in Section~\ref{u4i3tifgesec:bespoke} and,
after a cyclic relabelling, we can suppose that
\begin{equation}\label{51PRE}
d_0=\max \{ d_0,d_1,d_2\}.\end{equation}
For~$i\in \Z/3\Z$, we let 
\begin{equation}\label{eidef93047603998765EI}
e_i:=\sqrt{1-2\langle R_{i+1},R_{i+2}\rangle}.\end{equation}
We remark that~$e_i$ are well defined, in light of~\eqref{csinq}
and the fact that~$R_{i+1}$, $R_{i+2}\in\H$.

We also define
$$r_d:=\frac{4}{\gamma }\Big( -2\alpha (d_0d_1d_2-d_0-d_1-d_2) -2\epsilon \chi (d_0d_1+d_1d_2+d_2d_0-1)\Big) $$
and
\begin{equation}\label{DERI} r_i:=\frac{\vert r_d\vert}{d_{i+1}+d_{i+2}}.\end{equation}

We point out that,
with this notation, equation~\eqref{preq} reads 
as \begin{equation}\label{rieq.p}
e_{i+2}^2-e_{i+1}^2=r_d(d_{i+2}-d_{i+1}).\end{equation}
Moreover, we observe that for~$\epsilon =-1$, we have that~$r_d>0$. 
Therefore, in this case, $e_0$, $e_1$ and~$e_2$ are in the same order as~$d_0$, $d_1$ and~$d_2$.

The main calculation needed to understand repeated Napoleonization is as follows:

\begin{proposition}\label{REPEPRO}
We have that
\begin{equation}\label{rieq}
r_i\leq \rho := \frac{2}{3}+\frac{2}{27}+\frac{1}{3\sqrt{3}}\approx 0.93319.
\end{equation}
\end{proposition}

\begin{proof}
Applying the Cauchy-Schwarz inequality to the vectors~$(d_0,d_1,d_2)$ and~$(d_1,d_2,d_0)$, we see that
$$d_0d_1+d_1d_2+d_2d_0\leq d_0^2+d_1^2+d_2^2.$$
This observation and~\eqref{alphachigam} give that
\begin{equation}\label{suw4v83e2345678fbrekg}\begin{split}
\vert r_d\vert \leq\;& \frac{4}{\gamma }\Big( -2\alpha (d_0d_1d_2-d_0-d_1-d_2) +2\chi (d_0d_1+d_1d_2+d_2d_0-1)\Big)
\\ \leq\;& 
\frac{4}{\gamma }\Big( -2\alpha (d_0d_1d_2-d_0-d_1-d_2) +2\chi (d_0^2+d_1^2+d_2^2-1)\Big)
\\ \leq\;& 
\frac{4}{\gamma }\Big( -2\alpha (d_0d_1d_2-d_0-d_1-d_2)-4\alpha \chi \Big)\\ \leq\;& \frac{4}{\gamma }\Big( -2\alpha (d_0d_1d_2-d_0-d_1-d_2) \Big) +\frac{2}{3}.\end{split}\end{equation}

Furthermore, \begin{eqnarray*}
\frac{4}{\gamma }\Big( -2\alpha (d_0d_1d_2-d_0-d_1-d_2)\Big)
&=&\frac{4}{3}\frac{(d_0^2+d_1^2+d_2^2-1)(d_0d_1d_2-d_0-d_1-d_2)}{(d_0^2+1)(d_1^2+1)(d_2^2+1)}\\&\leq&
\frac{4}{3}\left( \frac{d_0^2+d_1^2+d_2^2-1}{d_0d_1d_2}\right) \\ &\leq& \frac{4}{3}\left( \frac{(d_1^2-1)(d_2^2-1)+d_1^2+d_2^2}{d_0d_1d_2}\right)\\&=&\frac{4}{3}\left( \frac{d_1^2d_2^2+1}{d_0d_1d_2}\right) ,\end{eqnarray*}
where the last inequality uses~\eqref{triineq}. 

Accordingly, using~\eqref{triineq.2},
we conclude that\begin{equation*}\displaystyle\frac{4}{\gamma }\Big( -2\alpha (d_0d_1d_2-d_0-d_1-d_2)\Big) \leq
\frac{4}{3}\left( \frac{d_1d_2}{d_0}+\frac{1}{3\sqrt{3}}\right).
\end{equation*}
We now observe that
\begin{eqnarray*}
2d_1d_2=d_1d_2+d_1d_2\le d_0\left( d_1+d_2\right),
\end{eqnarray*}
thanks to~\eqref{51PRE}, and therefore
\begin{equation*}\displaystyle\frac{4}{\gamma }\Big( -2\alpha (d_0d_1d_2-d_0-d_1-d_2)\Big) \leq
\frac{2}{3}\left( d_1+d_2\right)+\frac{4}{9\sqrt{3}} .
\end{equation*}
Plugging this information into~\eqref{suw4v83e2345678fbrekg}
we thereby find that
$$ |r_d|\le \frac{2}{3}\left( d_1+d_2\right)+\frac{4}{9\sqrt{3}}
+\frac{2}{3}.
$$

Consequently, recalling~\eqref{DERI} and using~\eqref{triineq.2}
and~\eqref{51PRE}, we have that, for any~$i\in \Z /3\Z$, 
$$
r_i\le \frac{1}{d_{i+1}+d_{i+2}}\left(
\frac{2}{3}\left( d_1+d_2\right)+\frac{4}{9\sqrt{3}}
+\frac{2}{3}\right)
\leq \frac{2}{3}+\frac{2}{27}+\frac{1}{3\sqrt{3}},$$
which gives~\eqref{rieq}, as desired.\end{proof}

For our purposes, Proposition~\ref{REPEPRO} is important, since it shows that
repeated Napoleonization, with~$\epsilon =\pm 1$, gives a sequence~$\{ P_0^{(k)}P_1^{(k)}P_2^{(k)}:k\geq 0\}$ of hyperbolic triangles, with geometrically decreasing differences of lengths of sides. This allows us to complete the proof of Theorem~\ref{REPENE}: in this respect, we distinguish the contractive iterations when~$\epsilon=1$ and when~$\epsilon=-1$.

\begin{proof}[Proof of Theorem~\ref{REPENE} when~$\epsilon =1$]
For~$\epsilon =1$,
we deduce from~\eqref{DEGAm}, \eqref{r1r2}
and~\eqref{eidef93047603998765EI} that
\begin{equation}\label{newpreqplus}\begin{split}
e_{i}^2-3&=1-2\langle R_{i+1}, R_{i+2}\rangle -3\\&=-2(1+\langle R_{i+1}, R_{i+2}\rangle)\\&=
-\frac{2}\gamma\Big( \gamma+
4(d_i^2+1)\big(\alpha d_{i+1}d_{i+2} + \chi  (d_{i+1}+d_{i+2})\big)
\\&\qquad\qquad - (d_i^2+1)(d_{i+1}^2-1)(d_{i+2}^2-1)+2(d_i^2-1)(d_i^2+1)\Big)\\&=
-\frac{2}{3(d_{i+1}^2+1)(d_{i+2}^2+1)}\Big( 3(d_{i+1}^2+1)(d_{i+2}^2+1)
+4 \alpha d_{i+1}d_{i+2}
\\&\qquad\qquad+4\chi  (d_{i+1}+d_{i+2}) -  (d_{i+1}^2-1)(d_{i+2}^2-1)+2(d_i^2-1) \Big)
\\&=\frac{2}{3(d_{i+1}^2+1)(d_{i+2}^2+1)}U_i,
\end{split}\end{equation}
where
\begin{eqnarray*}U_i&:=& 
-4\chi (d_{i+1}+d_{i+2})-4\alpha d_{i+1}d_{i+2}+(d_{i+1}^2-1)(d_{i+2}^2-1)\\&&\qquad
-2(d_{i}^2-1)-3(d_{i+1}^2+1)(d_{i+2}^2+1).\end{eqnarray*}

We also observe that, recalling~\eqref{firstalphaeq},
\begin{equation}\label{1214364ydhfbcbjith}\begin{split}
U_i \le\;&
-4\alpha d_{i+1}d_{i+2}+(d_{i+1}^2-1)(d_{i+2}^2-1)-2(d_{i}^2-1)-3(d_{i+1}^2+1)(d_{i+2}^2+1)
\\=\;&
-4\alpha (d_{i+1}d_{i+2}-1)
-2+2d_0^2+2d_1^2+2d_2^2 
+(d_{i+1}^2-1)(d_{i+2}^2-1)\\&\qquad\qquad-2(d_{i}^2-1)-3(d_{i+1}^2+1)(d_{i+2}^2+1)
\\=\;&-4\alpha (d_{i+1}d_{i+2}-1)-2(d_{i+1}^2+1)(d_{i+2}^2+1)\\=\;&
-2 \Big( (1-d_0^2-d_1^2-d_2^2) (d_{i+1}d_{i+2}-1)+(d_{i+1}^2+1)(d_{i+2}^2+1)\Big)\\
=\;& 
2 \Big( -d_{i+1}d_{i+2}+d_i^2d_{i+1}d_{i+2}+d_{i+1}^3d_{i+2}+d_{i+1}d_{i+2}^3-d_i^2\\&\qquad\qquad-2d_{i+1}^2
-2d_{i+2}^2 -d_{i+1}^2d_{i+2}^2
\Big)
\\=\;&
2\Big( (d_{i+1}d_{i+2}-1)(d_{i}^2-3)
+(d_{i+1}^2+d_{i+2}^2-3-d_{i+1}d_{i+2})\\&\qquad\qquad
+(d_{i+1}^2-d_{i+1}d_{i+2}+d_{i+2}^2)(d_{i+1}d_{i+2}-3)  \Big).
\end{split}\end{equation}

Now we set~$\mu _d:=\max \{ d_j^2-3:j\in \Z /3\Z\}$ and we claim that
\begin{equation}\label{wqitw4y8t21234567dsjkgfjpswr9t}\begin{split} &
(d_{i+1}d_{i+2}-1)(d_{i}^2-3)
+(d_{i+1}^2+d_{i+2}^2-3-d_{i+1}d_{i+2})
\\&\qquad+(d_{i+1}^2-d_{i+1}d_{i+2}+d_{i+2}^2)(d_{i+1}d_{i+2}-3)
\le ( d_{i+1}^2+d_{i+2}^2+1 )\mu _d.
\end{split}
\end{equation}
Indeed, if~$\mu_d=0$ then, in light of~\eqref{triineq.2}, we have that~$d_i=\sqrt{3}$ for all~$i\in \Z /3\Z$ and so
also the left-hand side
of~\eqref{wqitw4y8t21234567dsjkgfjpswr9t} equals zero. Therefore, in this case we are done. Hence, from now on we suppose that~$\mu_d\neq0$.

In this case, we point out that
\begin{eqnarray*}&&
d_{i+1}d_{i+2}-1
+\frac{d_{i+1}^2+d_{i+2}^2-3-d_{i+1}d_{i+2}}{\mu_d}
+d_{i+1}^2-d_{i+1}d_{i+2}+d_{i+2}^2\\&=&
\frac{d_{i+1}^2+d_{i+2}^2-3-d_{i+1}d_{i+2}-\mu_d}{\mu_d}
+d_{i+1}^2+d_{i+2}^2\\
&\le&\frac{d_{i+2}^2-d_{i+2}d_{i+1}}{\mu_d}
+d_{i+1}^2+d_{i+2}^2\\&\le&
\frac{d_{i+2}^2-3}{\mu_d}
+d_{i+1}^2+d_{i+2}^2\\&\le&1+d_{i+1}^2+d_{i+2}^2,
\end{eqnarray*}
where~\eqref{triineq.2} has also been used.
This implies~\eqref{wqitw4y8t21234567dsjkgfjpswr9t}, as desired.

Using the information in~\eqref{wqitw4y8t21234567dsjkgfjpswr9t}
into~\eqref{1214364ydhfbcbjith}, we obtain that
$$U_i\le 2(d_{i+1}^2+d_{i+2}^2+1 )\mu _d.$$
Hence, by~\eqref{newpreqplus}, 
\begin{equation}\label{mnbvcreReRTEiewoy}
e_i^2-3\leq \frac{4}{3}\left( \frac{
d_{i+1}^2+d_{i+2}^2+1}{d_{i+1}^2d_{i+2}^2+d_{i+1}^2+d_{i+2}^2+1}\right)\mu _d.\end{equation}

Now we point out that
$$
\frac{d_{i+1}^2+d_{i+2}^2+1}{d_{i+1}^2d_{i+2}^2}
=\frac1{d_{i+2}^2}+\frac1{d_{i+1}^2}+\frac{1}{d_{i+1}^2d_{i+2}^2}
\le\frac1{3}+\frac1{3}+\frac1{9}=\frac{7}{9}
,$$
thanks to~\eqref{triineq.2}, which gives that
\begin{equation*}
\frac{
d_{i+1}^2+d_{i+2}^2+1}{d_{i+1}^2d_{i+2}^2+d_{i+1}^2+d_{i+2}^2+1}\le
\frac{7}{16}.
\end{equation*}
{F}rom this and~\eqref{mnbvcreReRTEiewoy} we deduce that
$$e_i^2-3\leq \frac{7}{12}\mu _d.$$

Consequently, setting~$\mu _e:=\max \{ e_i^2-3:i\in \Z /3\Z \}$, we find that
\begin{equation}\label{contreq}
\mu _e\leq \frac{7}{12}\mu _d.
\end{equation}
Thus,
for~$\epsilon =1$ the hyperbolic triangles~$P_0^{(k)}P_1^{(k)}P_2^{(k)}$ satisfy
$$0\leq \max \big\{ (d_i^{(k)})^2-3:i\in \Z  /3\Z  \big\}\leq \left(\frac{7}{12}\right)^k\mu _d.$$ 
In particular, for~$ i=0,1,2$,
$$\lim_{k\rightarrow +\infty}d_i^{(k)}=\sqrt{3},$$
establishing the desired result.\end{proof}

\begin{proof}[Proof of Theorem~\ref{REPENE} when~$\epsilon =-1$]
For~$\epsilon =-1$, we have that~$r_d>0$. Accordingly, without loss of generality, we may suppose that~$e_{i_1}\leq e_{i_2}\leq e_{i_0}$, where~$d_{i_1}\leq d_{i_2}\leq d_{i_0}$, with~$i_0=0$ and~$\{ i_1,i_2\} =\{ 1,2\}$. 

By~\eqref{DEGAm}, \eqref{r1r2} and~\eqref{eidef93047603998765EI}, we have that
\begin{equation}\label{newpreq}\begin{split}
e_{0}^2-3&= 1-2\langle R_{1},R_{2}\rangle-3\\&=-2(1+\langle R_{1},R_{2}\rangle)\\&=
-\frac{2}\gamma\Big(\gamma+ 4(d_0^2+1)\big(\alpha d_{1}d_{2} - \chi  (d_{1}+d_{2})\big)\\&\qquad\qquad
- (d_0^2+1)(d_{1}^2-1)(d_{2}^2-1)+2(d_0^2+1)(d_0^2-1)\Big)\\&=\
-\frac{2}{3(d_1^2+1)(d_2^2+1)}\Big(3(d_1^2+1)(d_2^2+1)+ 4 \big(\alpha d_{1}d_{2} - \chi  (d_{1}+d_{2})\big)\\&\qquad\qquad
-  (d_{1}^2-1)(d_{2}^2-1)+2 (d_0^2-1)\Big)
\\&=\frac{2}{3(d_1^2+1)(d_2^2+1)}U_0,
\end{split}
\end{equation}
with
$$U_0:=4\chi (d_{1}+d_{2})-4\alpha d_{1}d_{2}+(d_{1}^2-1)(d_{2}^2-1)-2(d_{0}^2-1)-3(d_{1}^2+1)(d_{2}^2+1).$$

We remark that, using~\eqref{firstalphaeq},
\begin{equation}\label{fewt12345678kjhgfweur4y}\begin{split}
U_0=\;&
4\chi (d_{1}+d_{2})+4\alpha (1-d_{1}d_{2})-2+2d_0^2+2d_1^2+2d_2^2
\\&\qquad\qquad
+(d_{1}^2-1)(d_{2}^2-1)-2(d_{0}^2-1)-3(d_{1}^2+1)(d_{2}^2+1)\\
=\;&
4\chi (d_{1}+d_{2})+4\alpha (1-d_{1}d_{2})-2(d_{1}^2+1)(d_{2}^2+1)\\=\;&
4\chi (d_{1}+d_{2})+2(1-d_{0}^2-d_{1}^2-d_{2}^2)(1-d_{1}d_{2})-2(d_{1}^2+1)(d_{2}^2+1)\\=\;&
2\Big(2\chi (d_{1}+d_{2})-d_{0}^2-2d_{1}^2-2d_{2}^2-d_{1}d_{2}
+(d_{0}^2+d_{1}^2+d_{2}^2)d_{1}d_{2}-d_{1}^2d_{2}^2  \Big)\\=\;&
2\Big(2\chi (d_{1}+d_{2}) + (d_{1}d_{2}-1)(d_{0}^2-3)
+(d_{1}^2+d_{2}^2-3-d_{1}d_{2})\\&\qquad\qquad
+(d_{1}^2-d_{1}d_{2}+d_{2}^2)(d_{1}d_{2}-3)  \Big)\\ \leq\; &
2\Bigg( (d_{1}+d_{2})\sqrt{\frac{1}{3}\sum_{i=0}^2d_i^2(d_{i+1}^2-3)(d_{i+2}^2-3)} \\&\qquad\qquad+
(d_{1}d_{2}-1)(d_{0}^2-3)+
(d_{1}^2+d_{2}^2-3-d_{1}d_{2})
\\&\qquad\qquad+ (d_{1}^2-d_{1}d_{2}+d_{2}^2)(d_{1}d_{2}-3) \Bigg) ,
\end{split}\end{equation}
where~\eqref{chiest} is used for the inequality. 

Also, we claim that
\begin{equation}\label{du947593tgfqgdquyWWWWyasjwgqeiuq}
\begin{split}&
(d_{1}d_{2}-1)(d_{0}^2-3)+
(d_{1}^2+d_{2}^2-3-d_{1}d_{2})
+ (d_{1}^2-d_{1}d_{2}+d_{2}^2)(d_{1}d_{2}-3)\\&\qquad\qquad\le
(d_1^2+d_2^2)(d_0^2-3).
\end{split}
\end{equation}
Indeed, if~$d_0=\sqrt{3}$, then~\eqref{triineq.2} and~\eqref{51PRE}
give that also~$d_1=d_2=\sqrt{3}$, and so
the left-hand side of~\eqref{du947593tgfqgdquyWWWWyasjwgqeiuq}
vanishes as well, thus establishing
the desidered inequality.

If instead~$d_0>\sqrt{3}$, we recall~\eqref{51PRE} and we see that
\begin{eqnarray*}
&&d_{1}d_{2}-1+
\frac{d_{1}^2+d_{2}^2-3-d_{1}d_{2}}{d_0^2-3}
+ d_{1}^2-d_{1}d_{2}+d_{2}^2\\&=&
\frac{d_{1}^2+d_{2}^2-3-d_{1}d_{2}-d_0^2+3}{d_0^2-3}
+ d_{1}^2+d_{2}^2\\&=&
\frac{\min\{d_{1}^2,d_{2}^2\}+\max\{d_{1}^2,d_{2}^2\}
-d_{1}d_{2}-d_0^2}{d_0^2-3}
+ d_{1}^2+d_{2}^2\\&\le&
\frac{\min\{d_{1}^2,d_{2}^2\}-d_{1}d_{2}}{d_0^2-3}
+ d_{1}^2+d_{2}^2\\
&=&\frac{\min\{d_{1},d_{2}\}\left(\min\{d_{1},d_{2}\}
-\max\{d_{1},d_{2}\}\right)}{d_0^2-3}
+ d_{1}^2+d_{2}^2\\&\le&d_{1}^2+d_{2}^2,
\end{eqnarray*}
which implies the desired inequality in~\eqref{du947593tgfqgdquyWWWWyasjwgqeiuq} in this case as well.

{F}rom~\eqref{fewt12345678kjhgfweur4y}
and~\eqref{du947593tgfqgdquyWWWWyasjwgqeiuq},
and recalling also~\eqref{51PRE}, we thereby conclude that
\begin{eqnarray*}&&U_0\leq 2\left( (d_{1}+d_{2})\sqrt{\frac{1}{3}\sum_{i=0}^2d_i^2} +
d_1^2+d_2^2\right)(d_0^2-3)\\&&\qquad\qquad\le
2\left( 2d_{0} \sqrt{\frac{1}{3}\sum_{i=0}^2d_0^2} +
2d_0^2\right)(d_0^2-3)
\leq 8d_0^2(d_0^2-3).\end{eqnarray*}
This and~\eqref{newpreq} entail that 
\begin{equation}\label{ratest}e_{0}^2-3\leq
\frac{16 d_0^2 (d_0^2-3)}{3(d_1^2+1)(d_2^2+1)}.
\end{equation}

Now, writing
$$D_i^{(k)}:=(d_i^{(k)})^2:=1-2\langle P_{i+1}^{(k)},P_{i+2}^{(k)}\rangle,$$ we have that$$3\leq D_1^{(k)},D_2^{(k)}\leq D_0^{(k)}.$$ 
With this notation,
we deduce from~\eqref{rieq.p} and~\eqref{rieq} that
\begin{equation}
\label{newrieq}D_0^{(k+1)}-D_j^{(k+1)}\leq \rho (D_0^{(k)}-D_j^{(k)})\qquad \hbox{for}\quad j=1,2,\end{equation}
where~$\rho \in (0,1)$ is independent of~$k$, and also independent of~$P_0^{(0)}P_1^{(0)}P_2^{(0)}:=P_0P_1P_2$. 

Moreover~\eqref{ratest} becomes \begin{equation}
\label{newratest}D_0^{(k+1)}-3\leq \frac{16D_0^{(k)}(D_0^{(k)}-3)}{3(D_1^{(k)}+1)(D_2^{(k)}+1)}.
\end{equation}

By~\eqref{newrieq}, we can choose~$k_0$ so large that
$$\delta _0:=\rho ^{k_0}\max\Big\{ D_0^{(0)}-D_1^{(0)}, D_0^{(0)}-D_2^{(0)}
\Big\}<1.$$
In this way, we see that, for all~$k\geq k_0$,
$$ D_0^{(k)}- D_j^{(k)}\le \rho^{k+1}\big(
D_0^{(0)}- D_j^{(0)}\big)\le \rho^{k_0}\big(
D_0^{(0)}- D_j^{(0)}\big) \le\delta_0.
$$
Hence, by this and~\eqref{newratest}, for~$k\geq k_0$,
\begin{eqnarray*}&&D_0^{(k+1)}\leq 3+ \frac{16D_0^{(k)}(D_0^{(k)}-3)}{3(D_0^{(k)}+1-\delta _0)^2}\\&&\qquad=
3+\frac{16}{3}\left(
1-\frac{1-\delta _0}{D_0^{(k)}+1-\delta _0}\right)\left(1-\frac{4-\delta _0}{D_0^{(k)}+1-\delta _0}\right)<\frac{25}{3},\end{eqnarray*}
namely~$\beta _0:=25/3$ is an upper bound for the sequence~$\{ D_0^{(k)}:k>k_0\}$.

We will now iterate this argument as follows.
For~$p\in\N$, suppose that we have~$\delta _p\ge0$ and an upper bound~$\beta _p$ for the sequence~$\{ D_0^{(k)}:k>k_p\}$ for some~$k_p\in\N$. 
Notice that, in light of~\eqref{newratest}, the upper bound~$\beta _p\ge3$. 
Suppose also that~$\delta _p>0$ if~$\beta_p >3$, and~$\delta _p=0$ otherwise. 

If~$\beta _p=3$, set~$\delta _{p+1}:=0$, $\beta _{p+1}:=3$ and~$k_{p+1}:=k_p+1$.

If instead~$\beta _p>3$, by~\eqref{newrieq}, we can choose~$k_{p+1}>k_p$ so large that 
$$\delta _{p+1}:=
\rho ^{k_{p+1}}\max\Big\{ D_0^{(0)}-D_1^{(0)}, D_0^{(0)}-D_2^{(0)}
\Big\}
\leq \min \left\{ 1,\frac{\delta _p}{2},\beta _p+1-4\sqrt{\frac{\beta _p}{3}}\right\},$$ and we stress that
the right hand side is positive because~$\beta _p>3$. 

Also, by~\eqref{newratest}, for~$k\geq k_{p+1}>k_p$, we have that~$D_0^{(k)}\le \beta_p$, and therefore
\begin{eqnarray*}&&D_0^{(k+1)}\leq 3+\frac{16}{3}\left(
1-\frac{1-\delta _{p+1}}{D_0^{(k)}+1-\delta _{p+1}}\right)
\left(1-\frac{4-\delta _{p+1}}{D_0^{(k)}+1-\delta _{p+1}}\right)\\
&&\qquad\leq 3+\frac{16}{3}\left(1-\frac{1-\delta _{p+1}}{\beta _p+1-\delta _{p+1}}\right)\left(1-\frac{4-\delta _{p+1}}{\beta _p+1-\delta _{p+1}}\right)=:\beta _{p+1},\end{eqnarray*}
namely~$\beta _{p+1}$ is an upper bound for the sequence~$\{ D_0^{(k)}:k>k_{p+1}\}$. 

Also,
\begin{equation}\label{betaeq}\begin{split}
\beta _{p+1}-\beta _p=\;&\frac{16}{3}\cdot\frac{\beta _p (\beta _p-3)}{(\beta _p+1-\delta _{p+1})^2}-(\beta _p-3)\\=\;&
\frac{(\beta _p-3)}{3}\cdot\frac{16\beta _p -3(\beta _p+1-\delta _{p+1})^2}{(\beta _p+1-\delta _{p+1})^2}.\end{split}\end{equation}
We point out that, from the definition of~$\delta _{p+1}$,
we have that~$16\beta _p -3(\beta _p+1-\delta _{p+1})^2\le0$,
and thus~$\beta _{p+1}\le\beta _p$.
We conclude that the sequence of upper bounds~$\beta_p$ is
nonincreasing and bounded below by~$3$, and accordingly
the following limit exists:
$$\beta_\infty :=\lim_{p\rightarrow +\infty}\beta _p $$
with~$\beta_\infty\ge3$.

Moreover, evidently, $$\lim_{p\rightarrow+ \infty}\delta _p=0.$$
As a result, taking limits in~\eqref{betaeq}, 
$$0=\frac{(\beta_\infty-3)}{3}\cdot\frac{16\beta_\infty -3(\beta_\infty+1)^2}{(\beta_\infty+1)^2}=-\frac{(\beta_\infty-3)^2(3\beta_{\infty} -1)}{3(\beta_\infty+1)^2}.$$
This entails that~$\beta_\infty=3$.

Hence, for~$j=1,2$, we have that
$$3\leq \lim_{k\rightarrow+ \infty} D_j^{(k)}\leq  \lim_{k\rightarrow +\infty} D_0^{(k)}=3,$$
namely $$\lim_{k\rightarrow +\infty}d_i^{(k)}=\sqrt{3},$$
yielding the desired result.
\end{proof}

\section*{Acknowledgements}
This work has been supported by the Australian Future Fellowship FT230100333 and by the Australian Laureate Fellowship FL190100081.

\vfill

\end{document}